\documentclass[11pt]{article}
\usepackage{amsmath}
\usepackage{amsfonts}
\usepackage{amssymb}
\usepackage{mathrsfs}
\usepackage{latexsym}
\usepackage{cite}
\usepackage[dvips]{graphicx}
\usepackage{multirow}
\usepackage{setspace}
\usepackage{amsopn}
\usepackage{hyperref}
\allowdisplaybreaks
\newtheorem{theorem}{\bf Theorem}[section]
\newtheorem{lem}[theorem]{\bf Lemma}

\newtheorem{defn}[theorem]{\bf Definition}
\newenvironment{proof}{\noindent{\em Proof.}}{\quad \hfill$\Box$\vspace{2ex}}

\newenvironment{remark}{\noindent{\bf Remark}}{\vspace{2ex}}

\makeatletter % `@' now normal "letter"
\@addtoreset{equation}{section}
\makeatother  % `@' is restored as "non-letter"

\newcommand*{\dif}{\mathop{}\!\mathrm{d}}

\def \bR {\Bbb R}
\def \bN {\Bbb N}
\def \bZ {\Bbb Z}

\def \cG {\mathcal{G}}

\def \cF {\mathcal{F}}
\def \cL {\mathcal{L}}

\def \l {\left}
\def \r {\right}
\def \pa {\partial}

\def \lea {\lesssim}
\def \gea {\gtrsim}

\def \w {\widetilde}
\def \p {^\prime}
\def \ep {\epsilon}
\def \La {\Lambda_\lambda}

\begin{document}

\title{{\bf $L^2$ estimates of trilinear oscillatory integrals of convolution type on $\bR^2$}}

\author{Yangkendi Deng \thanks{School of Mathematical Sciences, University of Chinese Academy of Sciences, Beijing 100049, China. E-mail address: {\it dengyangkendi17@mails.ucas.edu.cn}.}\,\,\,,\quad Zuoshunhua Shi \thanks{School of Mathematics and Statistics, Central South University, Changsha, People's Republic of China. E-mail address: {\it shizsh@163.com}.}
\quad and \quad Dunyan Yan\thanks{School of Mathematical Sciences, University of Chinese Academy of Sciences, Beijing 100190, P. R. China. E-mail address: {\it ydunyan@ucas.ac.cn}.} }
	
\date{}
\maketitle

\noindent{{\bf Abstract}} \quad This paper is devoted to $L^2$ estimates for trilinear oscillatory integrals of convolution type on $\mathbb{R}^2$. The phases in the oscillatory factors include smooth functions and polynomials. We shall establish sharp $L^2$ decay estimates of trilinear oscillatory integrals with smooth phases, and then give $L^2$ uniform estimates for these integrals with polynomial phases.
\medskip

\noindent{{\bf MR(2020) Subject Classification}}  42B20 47G10


\section{Introduction} \label{introduction}

Consider the following trilinear oscillatory integrals of convolution type,
\begin{equation}\label{def}
\Lambda_\lambda(f_1,f_2,f_3)=\iint_{\bR^2} e^{i\lambda S(x,y)}f_1(x)f_2(y)f_3(x+y) \varphi(x,y) \dif x \dif y ,
\end{equation}
where $\lambda\in \bR$ is a parameter, $S$ is a real-valued smooth phase defined in a neighborhood of the origin, and $\varphi \in C_0^\infty(\bR^2)$ is a cut-off function near the origin.

If $S(x,y)\equiv 0$ and $\varphi$ is removed, then $\Lambda_\lambda(f_1,f_2,f_3)$ is equal to $\int_\bR f_1(x) (\w{f_2}\ast f_3(x))dx$, where $\w{f_2}(x)=f_2(-x)$ and $\w{f_2}\ast f_3$ is the convolution of $\w{f_2}$ and $f_3$. The boundedness of $\Lambda_\lambda$ in Lebesgue spaces is a consequence of Young's inequality.

Our purpose is to establish $L^2$ decay estimates of $\Lambda_\lambda$ when the phase function $S$ is non-degenerate. In other words, we are going to prove decay estimates of the following type:
\begin{equation}\label{decay estimate}
|\Lambda_\lambda(f_1,f_2,f_3)| \le C |\lambda|^{-\ep} \|f_1\|_2 \|f_2\|_2 \|f_3\|_2,
\end{equation}
where $\ep$ is a positive exponent, and $\|f_i\|_2$ is the $L^2$ norm of $f_i$.

A more general framework of multi-linear oscillatory integrals was studied in the seminal work by Christ, Li, Tao, and Thiele\cite{Christ-Li-Tao-Thiele}. It takes the form
\begin{equation}\label{general framework}
 \Lambda_\lambda(f_1,f_2,\cdots,f_n)=\int_{\bR^m}e^{i\lambda P(x)}\prod_{j=1}^{n}f_j(\pi_j(x))\eta(x) \dif x,
\end{equation}
where $\lambda\in \bR$ is a parameter, $P\in \bR[x_1, \cdots, x_m]$ is a real polynomial, $\pi_j : \bR^m \to V_j$ is a surjective linear mapping from $\bR^n$ onto some subspaces $V_j$ of $\bR^m$, and $\eta\in C_0^\infty(\bR^m)$ is a smooth cut-off function. In \cite{Christ-Li-Tao-Thiele}, all subspaces $V_j$ are assumed to have the same dimension. One of the main results in \cite{Christ-Li-Tao-Thiele} is the following theorem.
\begin{theorem}\label{theorem of general framework}
{\rm(\noindent\cite{Christ-Li-Tao-Thiele}~)}
Assume $n<2m$, $dimV_j=1$, and all $V_j$ lie in general position. If $P\in \bR[x_1, \cdots, x_m]$ is a real polynomial which is non-degenerate with respect to $\{\pi_j\}_{j=1}^n$, then
$$ |\Lambda_\lambda (f_1, f_2, \cdots, f_n)| \le C |\lambda|^{-\ep} \prod_{j=1}^{n} \|f_j\|_2, $$
where $\ep>0$ depends only on $n, m, \deg{P}$ and $\{\pi_j\}_{j=1}^n$. Moreover, this estimate is uniform if $degP$ is bounded above by a fixed number $d$, and the non-degenerate norm of $P$ has a uniform positive lower bound.
\end{theorem}

The trilinear oscillatory integral $\Lambda_\lambda(f_1,f_2,f_3)$ in (\ref{def}) is a special case of Theorem \ref{theorem of general framework}. The proof of Theorem \ref{theorem of general framework} in \cite{Christ-Li-Tao-Thiele} invokes the concept of $\lambda$-uniformity and reduces the multi-linear oscillatory integral to the trilinear one by induction on $n$.

Although a uniform decay estimate of $\Lambda_\lambda(f_1,f_2,f_3)$ has been established in Theorem \ref{theorem of general framework}, it seems that the sharp (uniform) $L^2$ decay estimate of $\Lambda_\lambda(f_1,f_2,f_3)$ may be of independent interest. In this paper, we shall prove the sharp $L^2$ decay estimates of trilinear $\Lambda_\lambda$ when the phase function $S$ is a real-valued smooth function, and then for polynomial phases satisfying certain non-degenerate assumptions, we will establish the uniform (also sharp) $L^2$ decay estimate for $\Lambda_\lambda$.

Assume $S(x,y)$ is a real-analytic function near the origin. %By Taylor's expansion, $S(x,y)=\sum_{\alpha,\beta\ge 0} c_{\alpha,\beta} x^\alpha y^\beta$ in a small neighborhood of the origin.
Let $H(x,y)= \pa_x \pa_y (\pa_x-\pa_y) S(x,y)$. By Taylor's expansion, we have $H(x,y)=\sum_{\alpha,\beta\ge 0} c_{\alpha,\beta} x^\alpha y^\beta$ in a small neighborhood of the origin. For simplicity, assume $H(x,y)$ is not identical to zero. Otherwise, if $H(x,y)\equiv 0$, then $S(x,y)$ is degenerate in the sense that $S(x,y)=p(x)+q(y)+r(x+y)$ for some smooth functions $p, q$ and $r$, and hence there is no decay for $\La$. Let $d$ be the order of $H$ at $(0,0)$, which is defined by
\begin{equation}\label{def of d}
d=\min\{\alpha+\beta \l| \alpha, \beta \in \bN, c_{\alpha, \beta}\ne 0 \r.\}.
\end{equation}

\noindent If $S(x,y)$ is smooth and $H(x,y)= \pa_x \pa_y (\pa_x-\pa_y) S(x,y)\ne 0$ on the support of $\varphi$, Li\cite{li} proved the following theorem.
\begin{theorem}\label{theorem of Li}
{\rm(\cite{li}~)}
Assume $S(x,y)$ is a real-valued smooth function near the origin. If $H(x,y)= \pa_x \pa_y (\pa_x-\pa_y) S(x,y)\ne 0$ on the support of $\varphi$, then the trilinear function $\La$ satisfies
$$ |\La(f_1,f_2,f_3)| \le C |\lambda|^{-\frac{1}{6}} \|f_1\|_2 \|f_2\|_2 \|f_3\|_2.$$
Moreover, if $\varphi(0,0) \ne 0$ then the above estimate is sharp.
\end{theorem}

In \cite{xiao}, Xiao proved the following sharp decay estimate with real-analytic phase functions.
\begin{theorem}\label{theorem of Xiao}
{\rm(\cite{xiao}~)}
Assume $S(x,y)$ is a real-valued real-analytic phase function, and $H(x,y)= \pa_x \pa_y (\pa_x-\pa_y) S(x,y)$ is a non-zero function. Then the trilinear functional $\La$ satisfies
$$ |\La(f_1,f_2,f_3)| \le C |\lambda|^{-\frac{1}{2(3+d)}} \|f_1\|_2 \|f_2\|_2 \|f_3\|_2, $$
where the order $d$ is given by {\rm(\ref{def of d})}. Moreover, this estimate is sharp provided that $\varphi(0,0)\ne 0$.
\end{theorem}

A question arises naturally whether Theorem \ref{theorem of Xiao} is still valid if the phase function $S$ is merely smooth. One of our purposes is to give a positive answer to this question. Assume $S(x,y)$ is smooth and $H(x,y)= \pa_x \pa_y (\pa_x-\pa_y) S(x,y)$ has a non-zero Taylor's expansion $H(x,y) \sim \sum_{\alpha,\beta \ge 0} c_{\alpha,\beta} x^\alpha y^\beta$. It should be pointed out that there is no decay for $\La$ if $\pa_x^\alpha \pa_y^\beta H(0,0)=0$ for all $\alpha, \beta \ge 0$. We say that $H$ is of finite type at $(0,0)$ if $\pa_x^\alpha \pa_y^\beta H(0,0)\ne 0$ for some non-negative integers $\alpha$ and $\beta$. Let $d$ be the order of $H$ at $(0,0)$, defined as in (\ref{def of d}), if $H$ is of finite type. Then one of our main results is
\begin{theorem}\label{Th-main}
Assume $S$ is a real-valued smooth function, and $H(x,y)= \pa_x \pa_y (\pa_x-\pa_y) S(x,y)$ is of finite type at the origin. Then the trilinear functional $\La$ satisfies
\begin{equation}\label{Th-main-eq}
|\La(f_1,f_2,f_3)|\le C |\lambda|^{-\frac{1}{2(3+d)}}  \|f_1\|_2\|f_2\|_2\|f_3\|_2.
\end{equation}
Moreover, this decay estimate is also sharp if $\varphi(0,0) \ne 0$.
\end{theorem}

Our proof of Theorem \ref{Th-main} is quite different from that of Theorem \ref{theorem of Xiao}. The difference lies in the resolution of singularities for $H(x,y)$. In the real-analytic case, resolution of singularities is a consequence of the Newton-Puiseux algorithm. For its applications to related topics, we refer the reader to Phong and Stein \cite{PS1997}, Greenblatt \cite{greenblatt1}, and Xiao \cite{xiao}. This method does not apply to a smooth phase function. Our resolution of singularities in this paper is due to Greenblatt\cite{greenblatt2}.

Away from the coordinate axes, we need resolution of singularities near $H(x,y)=0$ as in \cite{greenblatt2}. By the van der Corput Lemma and an almost orthogonality argument, the sharp $L^2$ decay estimate for one-dimensional oscillatory integral operators near the coordinate axes in \cite{greenblatt2} can be proved without resolution of singularities. However, we find the resolution of singularities for $H$ is also necessary to establish sharp $L^2$ decay estimate for $\La$ near the coordinate axes. In fact, this estimate cannot be proved by the same argument as in \cite{greenblatt2}.

On the other hand, we also consider uniform estimates of $\La$ when the phase function $S(x,y)$ is a polynomial. If some partial derivatives of $H(x,y)= \pa_x \pa_y (\pa_x-\pa_y) S(x,y)$ is bounded away from $0$, then we are able to establish uniform (also sharp) estimates for trilinear oscillatory integrals.
\begin{theorem}\label{Th-main-2}
Assume $S(x,y)\in \bR[x,y]$ is a real polynomial. Let $H(x,y)=\pa_x \pa_y (\pa_x -\pa_y) S(x,y)$. If $H(x,y)$ satisfies
\begin{equation}\label{lower bound of mixed derivatives}
|\pa^{\alpha^{(i)}} H(x,y)| \ge 1,~~ 1\le i \le N,~~ (x,y)\in [0,1]^2,
\end{equation}
where $\alpha^{(1)}, \alpha^{(2)}, \cdots, \alpha^{(N)}$ belong to $\bN^2$, and $\pa^{\alpha^{(i)}} H(x,y)= \pa^{\alpha_1^{(i)}}_x \pa^{\alpha_2^{(i)}}_y H(x,y)$ with $\alpha^{(i)}=\l(\alpha_1^{(i)}, \alpha_2^{(i)} \r)$, then for each cut-off function $\varphi \in C_0^\infty(\bR^2)$ with $supp(\varphi)\subseteq [0,1]^2$, there exists a constant $C$, depending only on $\deg{S}$ and $\varphi$, such that
$$ |\La(f_1,f_2,f_3)| \le C |\lambda|^{-\frac{1}{2(3+d)}} \|f_1\|_2 \|f_2\|_2 \|f_3\|_2,$$
where $d=\min\{|\alpha^{(i)}|: 1\le i \le N\}$.
\end{theorem}

Uniform estimates under assumptions of form (\ref{lower bound of mixed derivatives}) appeared in Carbery, Christ, and Wright\cite{carbery}, Phong, Stein, and Sturm\cite{PSS2001}, Carbery, and Wright\cite{CaW}, and Gressman \cite{gressman}. Our proof of Theorem \ref{Th-main-2} is inspired by the work of Phong, Stein, and Sturm\cite{PSS2001}.  For related works on this topic in this article, we refer the reader to
\cite{Christ2}, \cite{Christ3}, \cite{Christ4}, \cite{gressmanxiao}.
\section{Sharp estimates with smooth phases}
\subsection{Some lemmas}

In this section, we shall present some basic lemmas which will be frequently used in our analysis. Almost all of these results have appeared in literature cited below.

We begin with a simple size estimate for trilinear integrals of convolution type; see also \cite{xiao}.

\begin{lem} \label{size estimation}
Assume $R$ is a rectangle in $\bR^2$ with sides parallel to the axes, and it is of dimensions $\delta_1 \times \delta_2$. If $a \in L^\infty(\bR^2)$ is supported on $R$, then we have
$$ \l|\int_{\bR^2} f_1(x) f_2(y) f_3(x+y) a(x,y) \dif x \dif y \r| \le \min\{\delta_1^{\frac12},\delta_2^{\frac12}\} \|a\|_\infty \|f_1\|_2  \|f_2\|_2\|f_3\|_2.$$
\end{lem}

\noindent The following lemma is an almost orthogonality principle for trilinear integrals. This lemma has been used in \cite{xiao}.
\begin{lem}\label{almost orthogonality}
Suppose $T=\sum_l T_l$ is a trilinear functional, such that each $T_l$ can be written as
$$T_l (f_1,f_2,f_3)=\int \int K_l(x,y) f_1(x) f_2(y) f_3(x+y) \dif x\dif y,$$
where $T_l(x,y)$ is supported on a product of intervals $I_l \times J_l$.
Suppose there is a positive integer $L$ such that for a fixed $l$,
the number of $m$ with $I_l \cap I_m \ne \varnothing$ is bounded by $L$,
and the number of $m$ with $J_l \cap J_m \ne \varnothing$ is also bounded by $L$.
Then $\|T\|\le L \sup_l \|T_l\|$.
\end{lem}

\begin{proof}
By the assumption $T=\sum_lT_l$, we see that $|T(f_1,f_2,f_3)|$ is bounded by
\begin{eqnarray*}
 \sum_l \int \int \l|K_l(x,y) f_1(x) f_2(y) f_3(x+y)\r| \dif x\dif y
&\le & \sum_l \|T_l\| \|\chi_{I_l} f_1\|_2 \|\chi_{J_l} f_2\|_2 \|f_3\|_2 \\
&\le & \sup_l \|T_l\| \|f_3\|_2 \sum_l \|\chi_{I_l} f_1\|_2 \|\chi_{J_l} f_2\|_2.
\end{eqnarray*}
By Cauchy-Schwarz's inequality, it follows that
 $$\sum_l \|\chi_{I_l} f_1\|_2 \|\chi_{J_l} f_2\|_2
 \leq  \l( \sum_l \|\chi_{I_l} f_1\|_2^2 \r)^{1/2}
 \l( \sum_l \|\chi_{J_l} f_2\|_2^2  \r)^{1/2}
  \leq L \|f_1\|_2\|f_2\|_2.$$
This implies $\|T\|\le L \sup_l \|T_l\|.$
\end{proof}

\noindent The following lemma is contained in Christ \cite{Christ}.
\begin{lem}\label{sup estimate}
Assume $f$ is of class $C^n$ on some interval $I$ satisfying $|f^{(n)}(x)|\ge 1$ for all $x\in I$. Then there exists a constant $C=C(n)$ such that
$$ \sup_{x\in I} |f(x)| \ge C_n |I|^n .$$
\end{lem}

\noindent The following Bernstein lemma is due to Greenblatt\cite{greenblatt2}.

\begin{lem}\label{Bernstein-type inequality}
{\rm (\cite{greenblatt2}~)}
Let $f(x,y)$ be a smooth function in a neighborhood of the origin. Assume $(k,l)$ is a multiindex such that
\begin{equation}\label{Bernstein-type inequality-eq-1}
\dfrac{\pa^{k+l}f}{\pa x^k \pa y^l}(0,0) \ne 0.
\end{equation}
Suppose that $\delta,N >0 $. There is a neighborhood $U$ of the origin such that if $R \subset U$ is an $r_1$ by $r_2$ rectangle with $r_1 \le N r_2^\delta, r_2 \le N r_1^\delta$,
then for any multiindex $(\alpha,\beta)$ satisfying $0 \le \alpha,\beta \le 2$ there exists a constant $C$ such that
\begin{equation}\label{Bernstein-type inequality-eq-2}
\sup\limits_{(x,y)\in R}\l|\dfrac{\pa^{\alpha+\beta}f}{\pa x^\alpha \pa y^\beta}(x,y)\r| < Cr_1^{-\alpha} r_2^{-\beta} \sup\limits_{(x,y)\in R}|f(x,y)|.
\end{equation}
Furthermore, if $R\p$ is a subrectangle of $R$ of dimensions $\frac{r_1}{2}$ by $\frac{r_2}{2}$, then we have
\begin{equation}\label{Bernstein-type inequality-eq-3}
\sup\limits_{(x,y)\in R}|f(x,y)|<C \sup\limits_{(x,y)\in R\p}|f(x,y)|.
\end{equation}
\end{lem}

To give a uniform $L^2$ decay estimate for trilinear oscillatory integrals of convolution type, we need an operator van der Corput lemma. For its proof, see also Lemma \ref{van der Corput for trilinear} in \S3.
\begin{lem} \label{oscillatory estimation}
{\rm(\cite{xiao}~)}
Let $\La$ be defined as in {\rm(\ref{def})}. Assume the cut-off function $\varphi$ is supported in a rectangle $R$ of dimensions $\delta_1 \times \delta_2$ with $\delta_1>0, \delta_2>0$, and there exists a constant $A>0$ such that
$$ \delta_2 \sup_R |\pa_y a(x,y)|+ \delta_2^2 \sup_R |\pa_y^2 a(x,y)| \le A.$$
Let $S(x,y)$ be a real-valued smooth function, and let $H(x,y)=\pa_x \pa_y (\pa_x -\pa_y) S(x,y)$. If there exist two constants $\mu,B>0$ such that
$$ |H(x,y)| \ge \mu, \quad \delta_2 \sup_R |\pa_y H(x,y)|+ \delta_2^2 \sup_R |\pa_y^2 H(x,y)| \le B \mu ,$$
for all $(x,y)\in R$, then we have
$$ |\La(f_1,f_2,f_3)| \le C |\lambda \mu |^{-\frac16} \|f_1\|_2 \|f_2\|_2 \|f_3\|_2, $$
where the constant $C$ depends only on $A$ and $B$.

\end{lem}

\subsection{Resolution of singularities}\label{Resolution of singularities}

In this section, we are going to review the useful resolution of singularities due to Greenblatt\cite{greenblatt2}. It will play an important role in our proof of Theorem \ref{Th-main}.

Assume $H$ is a real-valued smooth function near the origin in $\bR^2$. Suppose $H$ is of finite type, i.e., there exist nonnegative integers $\alpha, \beta, (\alpha, \beta) \ne (0, 0)$, such that $\pa_x^\alpha \pa_y^\beta H(0,0) \neq 0$. For simplicity, we also assume $H(0,0)=0$.

Now we shall present resolution of singularities of $H$ away from the coordinate axes. Let $H(x,y) \sim \sum_{\alpha,\beta \ge 0} C_{\alpha,\beta} x^\alpha y^\beta$ be the formal Taylor's expansion of $H$. Then the Newton polyhedron $N(H)$ of $H$ is the convex hull of
$\bigcup_{\alpha,\beta\in \mathbb{N}}\{(x,y):x\geq \alpha,~y\geq\beta\}$, where the union is taken over all $\alpha,\beta$ such that $C_{\alpha,\beta}\neq 0$. The vertices of $N(H)$ are denoted by $(A_1, B_1), \cdots, (A_n, B_n)$, where $A_1<A_2<\cdots<A_n$ and $B_1>B_2>\cdots>B_n$. The case $n \ge 2$ will be considered since resolution of singularities is not necessary for $n=1$. This means that we assume $N(H)$ has at least two vertices throughout this section.

Let $-\frac{1}{M_i}$ be the slope of the line joining $(A_i, B_i)$ and $(A_{i+1}, B_{i+1})$. For any real number $c$, it is easy to see that
\[ H(x,cx^{M_i})=p_i(c)x^{d_i}+ o(x^{d_i}), x\to 0.\]
Here $p_i \in \bR[t]$ is a real polynomial. If $\Gamma_i$ denotes the compact face of $N(H)$ joining $(A_i,B_i)$ and $(A_{i+1},B_{i+1})$, then $p_i$ is given by
\begin{equation}\label{homogeneous polynomial}
 p_i(t)= \sum_{(\alpha, \beta)\in \Gamma_i} C_{\alpha, \beta} t^\beta.
 \end{equation}
One has $H(x,cx^{M_i})= Ax^d+ o(x^d)$ for some $A\ne 0$ unless $p_i(c)=0$. In contrast with the real analytic setting, we cannot obtain an algebraic curve $\gamma(x)$ by the Newton-Puiseux algorithm such that $\gamma(x)=Cx^{M_i}+o(x^{M_i})$ and $H(x,\gamma(x)) \equiv 0$. For a smooth phase, we shall follow Greenblatt's resolution of singularities.

Without loss of generality, we assume $p_i(c)=0$ and $c>0$. If $c$ is a complex root with ${\rm Im}(c)\ne 0$, then there $\gamma(x)=cx^{M_i}+o(x^{M_i})$ will not occur in the plane $\bR^2$. We consider $c>0$ since the treatment of $c<0$ is similar to that of $c>0$.

For a sufficiently small $\ep>0$, we shall restrict our attention to the following curved triangle:
$$ U= \{(x,y):x>0,~(c-\ep)x^{M_i}<y<(c+\ep)x^{M_i}\}.$$
If $j \in \bZ, j<0$ and $|j|$ is sufficiently large, we set
\begin{equation}\label{def Uj}
 U_j= \{(x,y): x\sim 2^j,~ (c-\ep)x^{M_i}<y<(c+\ep)x^{M_i} \},
\end{equation}
where $x\sim 2^j$ denotes $2^{j-1}\leq x \leq 2^{j+1}$. For a small number $\mu>0$, we can cover $U_j$ by finitely many dyadic squares with side length $\mu 2^{Mj}$, $M=\max\{M_i,1\}$, such that if $R$ is a dyadic square of dimension $\mu 2^{Mj} \times \mu 2^{Mj}$ and $R \bigcap U_j \ne \varnothing$, then its double $R^*\subseteq U_j^*$. Here $R^*$ has the same center as $R$, and its side length is twice as that of $R$. The domain $U_j^*$ is defined by
$$ U_j^*= \{(x,y): x\sim 2^j,~(c-2\ep)x^{M_i}<y<(c+2\ep)x^{M_i} \}.$$ The positive parameter $\ep$ is so small that the polynomial $p_i$ has no root other than $c$ in the interval $(c-2\ep, c+2\ep)$.

For convenience, we introduce some notations. If $R$ is a rectangle with sides parallel to the axes in $\bR^2$, we use $l_h(R)$ and $l_r(R)$ to denote the length of the horizontal side and the vertical side, respectively. For a square $R$, its side length is denoted by $l(R)$.

Let $\mathcal{D}$ be the family of all dyadic cubes in $\bR^2$. Now define
$$\cF =\{ R\in \mathcal{D} \l| l(R)= \mu 2^{Mj}, R \bigcap U_j \ne \emptyset \r. \}.$$
With the above preliminaries, we are going to describe the resolution of singularities for $H$ in $U_j$. This stopping time argument depends on whether the following inequality is true or not,
\begin{equation}\label{stopping time}
\sup_R |\nabla H(x,y)| l(R) < \frac14 \sup_R |H(x,y)|,
\end{equation}
 where $R$ is a dyadic square contained in some $\w{R} \in \cF$. For any $R \in \cF,$ if (\ref{stopping time}) holds for $R$, then the process stops. Otherwise, we dyadically divide $R$ into four equal dyadic squares. Then we will obtain a collection of dyadic squares, denoted by $\cF_1$. Each member $R$ of $\cF_1$ has side length either $l(R)=\mu 2^{Mj}$ or $l(R)=\mu 2^{Mj-1}$. For any $R \in \cF_1$, if the inequality (\ref{stopping time}) is true for $R$, then our procedure stops. Otherwise, we divide $R$ into four dyadic squares of side length $l(R)=\mu 2^{Mj-2}$. Let $\cF_2$ be the family of all dyadic squares obtained in this way.

By induction, if the family $\cF_n$ is given, we define $\cF_{n+1}$ be all dyadic squares $R$ satisfying one of the following conditions:\\

\noindent (i) $R\in \cF_n$ and the inequality (\ref{stopping time}) holds for $R$;\\
\noindent (ii) (\ref{stopping time}) does not holds
for some $R\p \in \cF_n$, and $R$ is a dyadic $\frac14$ subsquare in $R\p$.\\

\noindent Then a sequence $\{\cF_n\}$ of families of dyadic squares are obtained. Let $\cF_\infty$ be
 $$   \cF_\infty= \bigcap_{n=1}^\infty \cF_n.$$
 By the above stopping time procedure for resolution of singularities, we have
 $$ \l( \bigcup_{R\in \cF_\infty} R\r) \bigcap U_j= \l\{ (x,y)\in U_j \l| H(x,y) \ne 0 \r. \r\}.$$
 By the Bernstein inequality, i.e., Lemma \ref{Bernstein-type inequality}, we have
 $$ l(R)  \sup_R |\nabla H(x,y)| > \delta_0 \sup_R |H|, R \in \cF_\infty, $$
 for some constant $\delta_0>0$ independent of $R \in \cF_\infty$. Hence either
 \begin{equation}\label{Bernstein-type inequality case 1}
 l(R)  \sup_R \l|\frac{\pa H}{\pa x}(x,y)\r| >\frac{\delta}{2} \sup_R |H|
 \end{equation}
 or
  \begin{equation}\label{Bernstein-type inequality case 2}
 l(R)  \sup_R \l|\frac{\pa H}{\pa y}(x,y)\r| >\frac{\delta}{2} \sup_R |H|
 \end{equation}
 is true for all $0<\delta\leq \delta_0$. Write $R=I\times J$ for $R \in \cF_\infty$. Take $\delta>0$ to be sufficiently small. We shall expand each member $R \in \cF_\infty$ as large as possible such that $H$ stays within a factor of $2$ of a fixed constant on a larger rectangle $\w{R}$. For each $R \in \cF_\infty$, if both (\ref{Bernstein-type inequality case 1}) and (\ref{Bernstein-type inequality case 2}) are true for $R$, then set $\w{R}=R$, otherwise either (\ref{Bernstein-type inequality case 1}) or (\ref{Bernstein-type inequality case 2}) does not hold, say (\ref{Bernstein-type inequality case 1}) for example, let $\w{R}=\w{I} \times J$, where $\w{I}$ is the largest dyadic interval such that $\w{I} \supseteq I,~\w{R} \subseteq U_j^*$, and
 \begin{equation}\label{maximal interval-1}
 |\w{I}| \sup_{\w{R}} \l|\frac{\pa H}{\pa x}(x,y)\r| \le \delta\sup_{\w{R}} |H(x,y)|.
 \end{equation}
Similarly, if (\ref{Bernstein-type inequality case 2}) fails, then choose the largest dyadic interval $\w{J} \supseteq J, \w{R}= I\times \w{J} \subseteq U_j^*$ and
\begin{equation}\label{maximal interval-2}
|\w{J}| \sup_{\w{R}} \l|\frac{\pa H}{\pa y}(x,y)\r| \le \delta \sup_{\w{R}} |H(x,y)|.
\end{equation}
Then we can define a mapping from $\cF_\infty$ into the collection of coordinate rectangles $\mathcal{R}$,
\begin{align*}
 \alpha:~~ & \cF_\infty \longrightarrow \mathcal{R} \\
 & R \longmapsto \w{R}.
\end{align*}
Let $\cG$ be the range of $\alpha$, i.e.
$$ \cG= \l\{ \w{R}=\alpha(R) \l| R\in \cF_\infty \r. \r\}.$$
Now we list some important properties of rectangles in $\cG$.

\begin{theorem}\label{5-1}
{\rm(\cite{greenblatt2}~)}
There exists an exponent $\sigma>0$ and a constant $C>0$ such that for all $R=I\times J \in \cG$, we have
$$ |I|\le C|J|^\sigma, |J| \le C |I|^\sigma.$$
\end{theorem}

For $\ep>0$, let $R^{*,1+\ep}$ be the rectangle with the same center as $R$ and expanded by the factor $1+\ep$.

\begin{theorem}\label{5-2}
{\rm(\cite{greenblatt2}~)}
There exists a small number $\ep_0>0$, such that $\exists C_1, C_2>0$,
$$ C_1 \sup_{R^{*,1+\ep}} |H(x,y)| \le \inf_{R^{*,1+\ep}}|H(x,y)|,$$
$$  \sup_{R^{*,1+\ep}} |H(x,y)| \le C_2 \sup_{R}|H(x,y)| $$
for all $R \in \mathcal{G}$ and $0<\ep<\ep_0$.
\end{theorem}

To obtain a suitable partition of unity on $\l\{(x,y)\in U_j: H(x,y)\ne 0 \r\},$ one expects that appropriate dilations of rectangles in the collection $\cG$ have bounded overlap property. This is true even though members in $\cG$ come from various dyadic squares by performing a complicate procedure.

\begin{theorem} \label{5-3}
{\rm(\cite{greenblatt2}~)}
There exists an $\ep_0>0$ such that for $0<\ep<\ep_0$, the collection $\l\{ R^{*,1+\ep} \l| R\in \cG \r. \r\}$ has bounded overlap property in $U_j$, i.e., there exists an integer $ N \ge 1$ such that each point in $U_j$ is contained in at most $N$ rectangles $R^{*,1+\ep}$.
\end{theorem}
Moreover, the collection of rectangles $\l\{R^{*,1+\ep} \l| R\in \cG \r. \r\}$ has an appropriate mutual orthogonality. Let $\cG(k,m,n) \subseteq \cG $ consist of all $R\in \cG$ such that $R=I\times J$ satisfies \\
(i) $l_h(R)=|I|=\mu 2^{m+j}, l_v(R)=|J|=\mu 2^{n+j};$ \\
(ii) $2^{k+d_i j-1}\le \sup\limits_R |H(x,y)|< 2^{k +d_i j}$. \\
Then $\cG(k,m,n)$ has the following almost orthogonality property.

\begin{theorem}\label{5-4}
{\rm(\cite{greenblatt2}~)}
There exists a positive integer $N$ such that for all $k,m,n \in \bZ$ and any horizontal line or vertical line $L$, there are at most $N$ members in $\l\{R^{*,1+\ep} \l| R \in \cG(k,m,n) \r. \r\}$ having nonempty intersection with $L$.
\end{theorem}

\subsection{Estimates away from singularities and coordinate axes}

For simplicity, we begin with our decay estimates in a curved triangle domain away from singularities and coordinate axes. Let $H(x,y)=\pa_x \pa_y (\pa_x -\pa_y) S(x,y)$. Assume $(A_i,B_i)$ is the intersection of two compact edges of $N(H)$, denoted by $\Gamma_{i-1}$ and $\Gamma_i$. Let $p_{i-1}$ and $p_i$ be the polynomials associated with $\Gamma_{i-1}$ and $\Gamma_i$, respectively. It is more convenient to focus our attention in the first quadrant. The argument is similar in other quadrants.

Let $\delta_0>0$ be a small number. Assume $\Omega_i=\{(x,y)\in (0,\delta_0)^2 \l| x^{M_i} \lea y \lea x^{M_{i-1}} \r. \}$ is a curved triangle domain such that both $p_{i-1}$ and $p_i$ have no real roots in $\Omega_i$. Of course, the implicit constants, appearing in the definition of $\Omega_i$, are appropriately chosen.

By the symmetry of $x$ and $y$ in $\Lambda_{\lambda}$, we can further assume that $M_{i-1}\geq 1$. In fact, by a suitable partition of unity, we can restrict the integration of $\Lambda_{\lambda}$ in the domain
\begin{equation}\label{domain y less than x}
\Omega=\Big\{(x,y):x>0,~~0<y<Cx\Big\}
\end{equation}
with some positive constant $C$.

Choose a smooth function $\phi \in C_0^\infty (\frac12, 2)$ such that $\sum_{j\in \bZ} \phi(\frac{x}{2^j})=1 $ for all $x>0$. Let $\Lambda_{j,k}$ be defined as $\La$ in (\ref{def}), but with insertion of $\phi(\frac{x}{2^j})\phi(\frac{y}{2^k})$ into the cut-off function. Since our analysis is restricted to $\Omega_i$, we must have $M_i j +O(1) \le k \le M_{i-1} j + O(1)$, where $O(1)$ is a bounded constant independent of $j$ and $k$. By Lemma \ref{oscillatory estimation}, we obtain
$$ |\Lambda_{j,k}(f_1,f_2,f_3)| \le C \l(|\lambda| 2^{j A_i} 2^{k B_i}  \r)^{-\frac16} \|f_1\|_2 \|f_2\|_2 \|f_3\|_2 .$$
The size estimate in Lemma \ref{size estimation} gives
 $$ |\Lambda_{j,k}(f_1,f_2,f_3)| \le C \min\{ 2^{\frac{j}{2} },~ 2^{\frac{k}{2}} \} \|f_1\|_2 \|f_2\|_2 \|f_3\|_2 .$$

\noindent It follows from our assumption $M_i j +O(1) \le k \le M_{i-1} j + O(1)$ that $2^{j} \gea 2^{k}$. Hence we shall take $2^{\frac{k}{2}}$ as our upper bounds for size estimates of $\Lambda_{j,k}$. For clarity, we divide our proof into three cases.

\noindent{\bf Case (i) $2^{ A_i j + (B_i+3) M_{i-1} j} \lea |\lambda|^{-1}$}

\noindent In this case, we choose the size estimate for $\Lambda_{j,k}$. Then
\begin{eqnarray*}
  \sum_j \sum_{M_i j + O(1)\le k \le M_{i-1} j + O(1)} \|\Lambda_{j,k} \|
 &\lea & \sum_j \sum_{ k \le M_{i-1} j } 2^{\frac{k}{2}}\\
 &\lea & \sum_{j}  2^{\frac{M_{i-1}j}{2}} \\
 &\lea & |\lambda|^{-\frac{M_{i-1}}{2[ A_i+(B_i+3)M_{i-1}]}}.
\end{eqnarray*}

\noindent{\bf Case (ii) $2^{ A_i j+(B_i+3) M_{i} j} \gea |\lambda|^{-1}$}

\noindent The oscillation estimate is better than the size estimate.
This implies
\begin{eqnarray*}
 \sum_j \sum_{k} \|\Lambda_{j,k}\|
&\lea & \sum_j \sum_{M_i j \le k \le M_{i-1} j} \l( |\lambda| 2^{j A_i} 2^{k B_i}  \r)^{-\frac16} \\
&\lea & \sum_j \l( |\lambda| 2^{j A_i} 2^{j M_i B_i}  \r)^{-\frac16}\\
&\lea & |\lambda|^{-\frac{M_i}{2( A_i+ (B_i +3)M_i)}}.
\end{eqnarray*}

\noindent {\bf Case (iii) $2^{ A_ij+(B_i+3) M_{i} j} \lea |\lambda|^{-1}
\lea 2^{A_i j+(B_i+3) M_{i-1} j}$}

\noindent We denoted by $k_j$ the solution of
$ 2^{ A_i j+(B_i+3) k_j} = |\lambda|^{-1}.$ Then it is true that
\begin{equation*}
\sum_j \sum_{M_i j+ O(1) \le k \le k_j} \|\Lambda_{j,k}\|
\lea \sum_j  \sum_{M_i j+ O(1) \le k \le k_j} 2^{\frac{k}{2}}
\lea \sum_j 2^{\frac{k_j}{2}}.
\end{equation*}
On the other hand, for $k_j< k \le M_{i-1}j + O(1)$, we have
$$
\sum_j \sum_{k_j< k \le M_{i-1}j + O(1)} \|\Lambda_{j,k}\|
\lea \sum_j \sum_{M_i j+ O(1) \le k \le k_j}
\l(|\lambda| 2^{j A_i} 2^{k B_i} \r)^{-\frac{1}{6}}
\lea \sum_j 2^{\frac{k_j}{2}}.
$$
Let $j_1$ be the solution of $2^{A_i j_1+(B_i+3) M_{i-1} j_1}=|\lambda|^{-1}$. Then it follows that
$$\sum_{j\geq j_1} 2^{\frac{k_j}{2}}
\lea  \sum_j \l( |\lambda|^{}  2^{A_i j}  \r)^{-\frac{1}{2(B_i+3)}}
\lea  |\lambda|^{-\frac{M_{i-1}}{2(A_i+ (B_i+3)M_{i-1})}  }.
$$

We shall point out that the order $d$ in (\ref{def of d}) satisfies
$ d \geq \min\{A_i + M_i B_i,\frac{A_i}{M_i}+B_i\}$. This inequality is also true if $M_i$ is replaced by $M_{i-1}$. Hence the decay estimate in Theorem \ref{Th-main} is true if we restrict the integration of $\La$ over $\Omega_i$.

\subsection{Estimates near the singularity}\label{completion of the proof}
In this section, by Greenblatt's method of resolution of singularities in
\S \ref{Resolution of singularities}, $L^2$ decay estimates near the singularities of $H$ will be proved. Without loss of generality, we focus our attention in the first quadrant.

As in \S \ref{Resolution of singularities}, let $\Gamma_i$ denotes the compact face of $N(H)$ joining $(A_i,B_i)$ and $(A_{i+1},B_{i+1})$, and $p_i(t)=\sum_{(\alpha,\beta)\in\Gamma_i}C_{\alpha,\beta}t^{\beta}$. Recall that $-\frac{1}{M_i}$ is the slope of the straight line through $\Gamma_i$. Our argument is quite different depending on whether $M_i=1$ or not. Indeed, in the original coordinates, one is able to prove the sharp decay estimate in Theorem \ref{Th-main} only for $M_i=1$. However, the argument breaks down for $M_i\neq 1$. The reason is that a direct resolution of singularities does not exploit fully the convolution structure of $\Lambda_{\lambda}$, and that a logarithmic term $\log|\lambda|$ will appear in our estimate. Hence we cannot obtain the desired decay estimate in this way. For $M_i\neq 1$, it is necessary to make changes of variables before our application of resolution of singularities.

We begin with the case $M_i=1$. By insertion of a cut-off function, we define a trilinear operator $T(f_1,f_2,f_3)$ by
\begin{equation}\label{trilnear T}
 T(f_1,f_2,f_3)=\int \int e^{i\lambda S(x,y)}f_1(x)f_2(y)f_3(x+y) \varphi(x,y) \phi\l( \dfrac{y-\eta x }{\rho x}\r) \chi_{\{x,y>0\}}(x,y) \dif x \dif y ,
\end{equation}
where $\eta$ is a positive real zero of $p_i(t)$, $\rho$ is a small positive number, and $\chi_{\{x,y>0\}}$ is the characteristic function of the first quadrant. Let
$$ U= \{(x,y):x>0,~ (\eta-\rho/2)x<y<(\eta+2\rho)x\}.$$
It is easy to see that the cut-off function in $T(f_1,f_2,f_3)$ is compactly supported in $U$.

For a negative integer $j$, $|j|\gg 1$, we define
$$T_{j}(f_1,f_2,f_3)= \int \int_{\bR^2} e^{i\lambda S(x,y)}f_1(x)f_2(y)f_3(x+y) \varphi(x,y) \psi \l( \dfrac{y-\eta x }{\rho x}\r) \chi_{\{x,y>0\}} \psi_{j}(x) \dif x \dif y,      $$
where $\psi_j(x)=\phi(\frac{x}{2^j})$. Set
$$ U_j= \{(x,y):x\sim 2^j,~(\eta-\rho/2)x<y<(\eta+2\rho)x  \}.$$
This definition is slightly different from (\ref{def Uj}) in \S\ref{Resolution of singularities}.

As in \S\ref{Resolution of singularities}, by resolution of singularities for $H$ in $U_j$, we will obtain a collection of rectangles with dimensions $\mu 2^{m+j}$ by $\mu 2^{n+j}$, denoted by $\cG$. Here $\mu>0$ is a small fixed dyadic number, and $m,n$ are negative integers. Choose a nonnegative $C^\infty$ function $\tau(x)$ supported on $[-\frac{1}{2}-\frac{\ep}{2},\frac{1}{2}+\frac{\ep}{2}]$ such that $\tau(x)=1$ on $[-\frac{1}{2},\frac{1}{2}]$. Here $\ep=\frac{1}{2}\ep_0$, and $\ep_0$ is chosen as in Lemma \ref{5-2}. Assume $R$ is a $l_h(R)$ by $l_v(R)$ rectangle in $\cG$, and $(x_0,y_0)$ is its center.
Let $\tau_R(x,y)$ be defined by
\[ \tau_R(x,y)=\tau\l(\dfrac{x-x_0}{l_h(R)}\r)\tau\l(\dfrac{y-y_0}{l_v(R)}\r) .\]
By Lemma \ref{5-3}, we see that each $\xi_R(x,y)$ defined by
\[\xi_R(x,y)=\dfrac{\tau_R(x,y)}{\sum_{R\in \cG}\tau_R(x,y)},~R\in\cG,\]
is a smooth cut-off function in $C_0^{\infty}(R^{\ast,1+\epsilon_0})$. Moreover, $\{\xi_R:R\in\cG\}$ forms a partition of unity of $\{(x,y)\in U_j:
H(x,y)\neq 0\}$. Define a new cut-off function by
$$ \phi_{j,R}=\varphi(x,y) \psi \l( \dfrac{y-\eta x }{\rho x}\r) \chi_{x>0}(x,y) \psi_{j}(x) \xi_R(x,y).$$
Now we have the decomposition of $T_{j}=\sum_{R\in \cG}T_{j,R}$, where
\[ T_{j,R}(f_1,f_2,f_3)=\int \int e^{i\lambda S(x,y)}f_1(x)f_2(y)f_3(x+y) \phi_{j,R}(x,y) \dif x \dif y. \]

Using Lemma \ref{Bernstein-type inequality}, Theorems \ref{5-1}, \ref{5-2} and \ref{5-3}, we are able to prove the following upper bounds for $\phi_{j,R}$ and its partial derivatives; see also Lemma 4.8 in \cite{greenblatt2}.

\begin{lem}\label{lem6-1}
Suppose $\phi_{j,R}=\phi(x,y)\psi_j(x)\xi_R(x,y)$ for some $R\in \cG$ which is of dimensions $r_1$ by $r_2$. For $\alpha=0,1,2$, there is a constant $C>0$ such that
\[   \l|\dfrac{\pa^\alpha \phi_{j,R}}{\pa x^\alpha}(x,y) \r|< C r_1^{-\alpha}, ~~~~~\l|\dfrac{\pa^\alpha \phi_{j,R}}{\pa y^\alpha}(x,y) \r|< C r_2^{-\alpha} .\]
In addition, there exist a constant $W_{j,R}>0$ and a constant $\delta>0$ such that on $R$ we have
\[ \delta W_{j,R}  < |H(x,y)|< W_{j,R} .\]
Also, there exist $\delta\p, C\p >0$ such that we have
\begin{equation}\label{eq6-1}
  \delta\p W_{j,R}< r_1 \sup\limits_{(x,y)\in R} \l| \dfrac{\pa H}{\pa x}(x,y) \r|, ~~r_2 \sup\limits_{(x,y)\in R} \l| \dfrac{\pa H}{\pa y}(x,y) \r|;
\end{equation}
\[ r_1 \sup\limits_{(x,y)\in R} \l| \dfrac{\pa H}{\pa x}(x,y) \r|,~~ r_1^2 \sup\limits_{(x,y)\in R} \l| \dfrac{\pa^2 H}{\pa x^2}(x,y) \r|<C\p W_{j,R} ;\]
\[ r_2 \sup\limits_{(x,y)\in R} \l| \dfrac{\pa H}{\pa y}(x,y) \r|, ~~r_2^2 \sup\limits_{(x,y)\in R} \l| \dfrac{\pa^2 H}{\pa y^2}(x,y) \r|<C\p W_{j,R} .\]
Moreover, all above estimates are also true with $R^{*,1+\ep}$ in place of $R$, where $\epsilon>0$ is sufficiently small.
\end{lem}

Now we get to estimate $\|T\|$. By the almost orthogonality principle in Lemma \ref{almost orthogonality}, it follows from $T=\sum_j T_j$ that $\|T\| \lea \sup_j \|T_j\|$. Hence it suffices to estimate $\|T_j\|$ for fixed integer $j$. By our assumption $M_i=1$, one has $d_i=A_i+M_iB_i=A_i+B_i$ for each $(A_i,B_i)\in \Gamma_i$. There is a constant $C>0$ such that
\begin{equation}\label{eq6-3}
\sup_R |H(x,y)|< C 2^{d_ij} .
\end{equation}
Let $r\ge 1$ be the order of the zero $\eta$ of $p_i(t)$. This implies that $p_i(\eta)=\cdots=p_i^{(r-1)}(\eta)=0$ and $p_i^{(r)}(\eta)\neq0$. For $k,m,n \in \bZ_{<0}$, we define $\cG(k,m,n) \subseteq \cG $ as in Theorem \ref{5-4} in \S\ref{Resolution of singularities}. Then we have the following lemma; see also \cite{greenblatt2}.
\begin{lem} \label{lem6-2}
Suppose $T_{j,R}\in \cG(k,m,n)$. For some constants $c_1$ and $c_2$ we have
\[  c_1+k<m<c_2+\frac{k}{r}, ~~c_1+k<n<c_2+\frac{k}{r}. \]
\end{lem}
\begin{proof}
We prove only the inequality for $m$. Observe that $H$ satisfies
$\sup_R \l| \partial_x H(x,y) \r| \lea 2^{d_ij-j}. $
By Lemma \ref{lem6-1}, we have
$$ 2^{m+j} \sup_R \l| \dfrac{\pa H}{\pa x }(x,y) \r|  \gea \sup_R |H(x,y)| \sim 2^{k+jd_i}.$$
Hence there exists a constant $c_1$ such that $m\geq c_1+k.$ On the other hand,
$$ \inf_R\l| \dfrac{\pa^r H}{\pa x ^r}(x,y) \r|\gea 2^{-jr+jd_i}.$$
By Lemma \ref{sup estimate}, we get
$$ \sup_R |H(x,y)| \gea 2^{r(m+j)} \inf_R \l|\dfrac{\pa^r H}{\pa x ^r}(x,y) \r| \gea 2^{mr+jd_i}.$$
Combining with $  \sup_R |H(x,y)| \sim 2^{k + jd_i},$
we see that $ m\leq c_2+\frac{k}{r}$ for some constant $c_2$.
\end{proof}

Observe that $T_j=\sum_{k\leq 0} \sum_{c_1+k<m,n<c_2+k/r} \sum_{R \in \cG(k,m,n)} T_{j,R}$. In view of the orthogonality property in Theorem \ref{5-4}, we see that
\begin{equation}\label{eq6-4}
\|T_j\| \lea \sum_{k\leq 0} \sum_{c_1+k<m,n<c_2+\frac{k}r} \sup_{R \in \cG(l,m,n)} \|T_{j,R}\|.
\end{equation}

\noindent We use Lemmas \ref{size estimation} and \ref{oscillatory estimation} to bound $\|T_{j,R}\|$. More precisely, by the size estimate in Lemma \ref{size estimation}, one has
$$\|T_{j,R}\|\lea \min(2^{\frac{m+j}{2}},2^{\frac{n+j}{2}}).$$
On the other hand, the oscillation estimate in Lemma \ref{oscillatory estimation} gives us the following bound,
$$\|T_{j,R}\|\lea |\lambda|^{-\frac{1}{6}} 2^{-\frac{k+jd_i}{6}}.$$
By (\ref{eq6-4}), it follows immediately that
\begin{equation}\label{eq6-5}
 \|T_j\|\lea \sum\limits_{k \leq 0}\sum\limits_{c_1+k<m,n<c_2+\frac{k}r} \min\l( 2^{\frac{m+j}{2}},2^{\frac{n+j}{2}},|\lambda|^{-\frac{1}{6}} 2^{-\frac{k+jd_i}{6}}  \r) .
\end{equation}
We fix a constant $c_3>\max(1, c_2-c_1)$. Observe that $c_1$ is a constant, so the right hand of (\ref{eq6-5}) is not greater than a constant multiple of
\begin{equation}\label{eq6-6}
\sum\limits_{k \leq 0}\sum\limits_{k<m,n<c_3+\frac{k}{r}} \min\l( 2^{\frac{m+j}{2}},2^{\frac{n+j}{2}},|\lambda|^{-\frac{1}{6}} 2^{-\frac{k+jd_i}{6}}  \r).
\end{equation}
Without loss of generality, we assume $|\lambda|\geq 2$. Then we consider the equation
\begin{equation}\label{eq6-7}
\frac{1}{2}\left(\frac{k^\prime}{r}+j^\prime \right)=-\frac{1}{6}\log_2|\lambda|-\frac{1}{6}\l(k^\prime+j^\prime d_i\r).
\end{equation}
Let $j_0$ be $j^\prime$ in (\ref{eq6-7}) when $k^\prime=0$. However, this solution $j_0$ is not necessarily an integer. The definition of $j_0$ implies $2^{j_0}=|\lambda|^{-1/(3+d_i)}$. Now we divide our argument into two cases $j\leq  j_0$ and $j> j_0$.

If $j \leq j_0$, we take the size estimate in (\ref{eq6-7}). It follows that
\begin{align}
\|T_j\| &\lea \sum\limits_{k \leq 0}\sum\limits_{k<m,n<c_3+\frac{k}{r}} \min\l( 2^{\frac{m+j}{2}},2^{\frac{n+j}{2}}\r) \notag\\
&\le \sum\limits_{k \leq 0}\sum\limits_{k<m,n<c_3+\frac{k}{r}} 2^{\frac{m+n+2j}{4}}  \notag\\
&\lea |\lambda|^{-\frac{1}{2(3+d_i)}}. \label{eq6-8}
\end{align}

Now we consider $j>j_0$. Let $k_j$ be the solution of $k^\prime$ in (\ref{eq6-7}) when $j^\prime=j$. For simplicity, we shall divide the summation in (\ref{eq6-6}) into the following two parts:
\begin{eqnarray*}
I_1&:=&\sum\limits_{k\geq k_j}\sum\limits_{k<m,n<c_3+\frac{k}{r}}
\min\l( 2^{\frac{m+j}{2}},2^{\frac{n+j}{2}},|\lambda|^{-\frac{1}{6}} 2^{-\frac{k+jd_i}{6}}  \r),\\
I_2&:=&\sum\limits_{k < k_j}\sum\limits_{k<m,n<c_3+\frac{k}{r}}
\min\l( 2^{\frac{m+j}{2}},2^{\frac{n+j}{2}},|\lambda|^{-\frac{1}{6}} 2^{-\frac{k+jd_i}{6}}  \r).
\end{eqnarray*}
For $I_2$, we apply the size estimate to obtain
\begin{eqnarray*}
I_2=\sum\limits_{k < k_j}\sum\limits_{k<m,n<c_3+\frac{k}{r}}
 2^{\frac{m+n+2j}{4}}
 \lea 2^{\frac{j}2+\frac{k_j}{2r}}.
\end{eqnarray*}
Note that $(k',j')=(j,k_j)$ is a solution of the equation (\ref{eq6-7}). This implies
$$\frac{r+3}{r}k_j+(d_i+3)j=-\log|\lambda|.$$
Hence
\begin{eqnarray*}
I_2&\lea& 2^{\frac{j}{2}-\frac{d_i+3}{2(r+3)}j} |\lambda|^{-\frac{1}{2(r+3)}}\\
&\lea& 2^{\frac{r-d_i}{2(r+3)}j} |\lambda|^{-\frac{1}{2(r+3)}}\\
&\lea& 2^{\frac{r-d_i}{2(r+3)}j_0} |\lambda|^{-\frac{1}{2(r+3)}}\\
&\lea& |\lambda|^{-\frac{1}{2(d_i+3)}},
\end{eqnarray*}
where we have used the facts $r\leq d$, $j<j_0$, and $2^{j_0}=|\lambda|^{-\frac{1}{(d_i+3)}}$.

Now we turn to $I_1$. Consider the following inequality:
\[ \frac{1}{4}(m+n+2j)\ge -\frac{1}{6}\log_2|\lambda|-\frac{1}{6}(k+jd_i).\]
Observe that the equation (\ref{eq6-7}) holds for $k^\prime=k_j, j^\prime=j$. By subtracting this equation from the above inequality,
we have
\[ \l(m-\frac{k}{r}\r)+ \l(n-\frac{k}{r}\r) \geq \l(\frac{2}{3}+\frac{2}{r}\r)(k_j-k) .\]
Let
$$ A=\l\{ (m,n)\in \bZ_{<0}^2 : m\leq \frac{k}{r},~n\leq \frac{k}{r} \quad \text{and} \quad \l(m-\frac{k}{r}\r)+ \l(n-\frac{k}{r}\r) \geq  \l(\frac{2}{3}+\frac{2}{r}\r)(k_j-k) \r\},$$
$$ B=\l\{ (m,n)\in \bZ_{<0}^2 : m\leq \frac{k}{r},~n\leq \frac{k}{r} \quad \text{and} \quad \l(m-\frac{k}{r}\r)+ \l(n-\frac{k}{r}\r) < \l(\frac{2}{3}+\frac{2}{r}\r)(k_j-k) \r\}.$$
It is easy to see that, the number of elements in $A$, denoted by $\# A$, is comparable to $(k-k_j)^2$. On the other hand, for $t<0$ which satisfies $t+\frac{2k}{r}\in \bZ_{<0}$, we define
$$ B_t=\l\{ (m,n)\in \bZ_{<0}^2 : m\leq \frac{k}{r},~n\leq \frac{k}{r} \quad \text{and} \quad \l(m-\frac{k}{r}\r)+ \l(n-\frac{k}{r}\r) =t \r\}.$$
It is clear that $ \#B_t \le |t|+1$ and
$$ B=\bigcup\limits_{t} B_t,$$
where the union is taken over all $t$ satisfying $t+\frac{2k}{r}\in \bZ_{<0}$ and $t\leq(\frac{2}{3}+\frac{2}{r})(k_j-k)$.
So, by calculation and definition of $k_j$, we have
\begin{eqnarray*}
I_1&=&\sum\limits_{k_j\leq k \leq 0}
\sum\limits_{k<m,n<c_3+\frac{k}{r}}\min\l( 2^{\frac{m+j}{2}},2^{\frac{n+j}{2}},|\lambda|^{-\frac{1}{6}} 2^{-\frac{k+jd_i}{6}} \r) \notag \\
&\le & \sum\limits_{ k \geq k_j}\sum\limits_{m,n\leq c_3+\frac{k}{r}}\min\l( 2^{\frac{m+n+2j}{4}}, |\lambda|^{-\frac{1}{6}} 2^{-\frac{k+jd_i}{6}}\r) \notag \\
&\le & \sum\limits_{ k \geq k_j}\sum\limits_{(m,n)\in A}|\lambda|^{-\frac{1}{6}} 2^{-\frac{k+jd_i}{6}}+
\sum\limits_{ k \geq k_j}\sum\limits_{(m,n)\in B } 2^{\frac{m+n+2j}{4}}\notag :=I_{1,1}+I_{1,2}. \label{eq6-10}
\end{eqnarray*}
For $I_{1,1}$, it follows from $\# A\lea 1+(k-k_j)^2$ that
\begin{eqnarray*}
I_{1,1}
&\lea & \sum\limits_{ k \geq k_j} (1+(k_j-k)^2 )|\lambda|^{-\frac{1}{6}} 2^{-\frac{k+jd_i}{6}}\\
&\lea & |\lambda|^{-\frac{1}{6}} 2^{-\frac{k_j+jd_i}{6}}\\
&=& 2^{\frac{\frac{k_j}{r}+j}{2}},
\end{eqnarray*}
where the last equality is true since $(k',j')=(k_j,j)$ is a solution to (\ref{eq6-7}). For $I_{2,2}$, we have
\begin{eqnarray*}
I_{1,2} &\lea & \sum\limits_{ k \geq k_j} \sum\limits_{t+\frac{2k}{r}\in \bZ_{<0}, t\leq (\frac{2}{3}+\frac{2}{r})(k_j-k)}  \sum\limits_{(m,n)\in B_t } 2^{\frac{m+n+2j}{4}} \\
&\lea & \sum\limits_{ k \geq k_j} \sum\limits_{t+\frac{2k}{r}\in \bZ_{<0}, t\leq (\frac{2}{3}+\frac{2}{r})(k_j-k)}  (|t|+1) 2^{\frac{t+\frac{2k}{r}+2j}{4}}\\
&\lea & 2^{\frac{\frac{k_j}{r}+j}{2}}.
\end{eqnarray*}
By the same argument as in the estimate of $I_2$, we see that $I_1$ is bounded by a constant multiple of $|\lambda|^{-\frac{1}{2(d_i+3)}}$.\\

Now we turn to the case $M_i>1$. We have pointed out that the desired decay estimate cannot be obtained by resolution of singularities in the original coordinates $x,y$. For this reason, we shall make a change of variables in the following domain:
\begin{equation}\label{Domain Omega eps}
\Omega_{\epsilon}=\Big\{(x,y):x>0,~|y|\leq 2\epsilon x\Big\}
\end{equation}
where $\epsilon>0$ is a suitably small number.

If there exists some $M_i=1$, then we choose a small number $\epsilon>0$ such that $p_i(t)$, defined by (\ref{homogeneous polynomial}), has no non-zero roots in $(-2\epsilon,2\epsilon)$. Otherwise let $\epsilon=1$ if there is no $M_i=1$. Choose a smooth function $a(x)\in C_0^{\infty}(-2,2)$ such that $a(x)=1$ on $[-1,1]$. First, we insert $a\l(\frac{y}{\epsilon x}\r)\chi_{\{x>0\}}$ into the cut-off function of $\Lambda_{\lambda}$. We shall make the following changes of variables:
\begin{equation}\label{changes of variabels}
u=x,~~~v=x+y.
\end{equation}
It is easy to see that $\widetilde{H}(u,v)=H(x,y)$ has the same generalized order $\widetilde{d}$ as $H$, i.e., $\widetilde{d}=d$.  The trilinear oscillatory integral $\Lambda_{\lambda}$ becomes
\begin{equation}
\widetilde{{\Lambda}}_\lambda(f_1,f_2,f_3)=\iint_{\bR^2} e^{i\lambda S(u,v-u)}f_1(u)f_2(v-u)f_3(v) \varphi(u,v-u)a\l( \frac{v-u}{\epsilon u} \r)\chi_{\{u>0\}}\dif u \dif v .
\end{equation}
This implies that $\widetilde{{\Lambda}}_{\lambda}$ is also a trilinear oscillatory integral of convolution type. On the other hand, the curved box $\Omega_{\epsilon}$ in (\ref{Domain Omega eps}) becomes
\begin{equation}
\widetilde{ \Omega_{\epsilon}}= \Big\{(u,v): u>0,~ u-2\epsilon u<v<u+2\epsilon u \Big\}.
\end{equation}
If we apply the method of resolution of singularities for $\widetilde{H}(u,v)$, then it is suffices to consider the corresponding compact face of $N(\widetilde{H})$ with slope $-1$. Hence our above argument works.

Combining all above results, we have completed the proof of Theorem \ref{Th-main}.\\

\begin{remark}
The decay estimate in Theorem \ref{Th-main} and the Newton polyhedron of $H$ are unrelated. In fact, during our proof, it suffices to treat only the compact face of $N(H)$ with slope $-1$.
\end{remark}

\section{$L^2$ uniform estimates for trilinear oscillatory integrals}
In this section, we shall establish $L^2$ uniform estimates for the trilinear oscillatory integrals $\Lambda_{\lambda}(f,g,h)$ with polynomial phases. The resulting $L^2$ estimate will be a consequence of an operator van der Corput lemma and uniform estimates for a class of sublevel set operators. We begin with a useful decomposition of an algebraic domain, due to Phong, Stein, and Sturm \cite{PSS2001}, which will play an important role in our argument.

\begin{defn}\label{curved trapezoid}
Assume $a$ and $b$ are real numbers, $a<b$. Suppose $g$ and $h$ are continuous monotone functions on $[a,b]$, and $g(x)<h(x)$ for all $a<x<b$. Then the domain
$$ \Omega=\l\{ (x,y)\in \bR^2: a<x<b, g(x)<y<h(x) \r\}$$
is called a curved trapezoid.
\end{defn}

\begin{defn}\label{algebraic domain}
We say that $D\subseteq U=[0,1]^d$ is a simple algebraic domain of type $(r,n)$ if $D$ can be written as
$$ D=\l\{ (x_1, \cdots, x_d)\in U \l| P_k(x)\ge \lambda_k, 1\le k \le r\p \r. \r\},$$
where $P_k\in \bR[x_1, \cdots, x_d]$ are real polynomials, $\lambda_k\in \bR, 1\le r\p \le r$, and $\deg P_k \le n$.

We say that $D$ is an algebraic domain of type $(r, n, d, \omega)$ if $ D=\bigcup_{i=1}^{\omega\p}D_i$ with $\omega\p \le \omega$, where each $D_i$ is a simple algebraic domain of type $(r,n)$.
\end{defn}

Phong, Stein, and Sturm proved the following theorem for algebraic domains. Roughly speaking, any algebraic domain can be decomposed into finitely many disjoint curved trapezoids, up to a set of measure zero.

\begin{lem}\label{decomposition of algebraic domains}
{\rm (\cite{PSS2001}~)}
Let $D$ be an algebraic domain of type $(r, n, 2, \omega)$. Then there exists finitely many curved trapezoids $\Omega_1, \Omega_2, \cdots, \Omega_M$, and a set $Z$ of measure zero, such that
$$ D=\l(\bigcup_{i=1}^M \Omega_i \r) \bigcup Z.$$
Moreover, the number $M$ is bounded in terms of $r, n, \omega$.
\end{lem}

For our application, we shall give an outline of proof of Lemma \ref{decomposition of algebraic domains}. For simplicity, we assume $D$ is a simple algebraic domain. The treatment for general algebraic domain is essentially the same as that for a simple algebraic domain. Then
$$ D=\l\{ (x,y)\in [0,1]^2 \l| P_k(x,y) \ge \lambda_k, 1\le k \le r \r. \r\}.$$
Let $Q_k=P_k-\lambda_k$ and decompose $Q_k$ as
\begin{equation}\label{decomposition of qk}
Q_k(x,y)=\prod_{l} \l(P_{k,l}(x,y) \r)^{m_l},
\end{equation}
where $Q_{k,l}$ are irreducible polynomials in $\bR[x,y]$, $m_l \ge 1$, and $P_{k,l_1}, P_{k,l_2}$ are relatively prime for $l_1 \ne l_2$. Let
$$ \Gamma_1=\bigcup \l\{(x,y)\in \bR^2 \l| P_{k,l}(x,y) =0 \r. \r\},$$
where the union is taken over all $P_{k,l}$ such that $P_{k,l}$ is a polynomial independent of $x$ or $y$. In other words, each $Q_{k,l}$ is a polynomial in only one variable. It is clear that $\Gamma_1$ consists of finitely many horizontal (vertical) straight lines. Of course, it is possible that $\Gamma_1=\emptyset$.

Define $\Gamma_2$ by
$$ \Gamma_2=\bigcup \l\{ (x,y)\in \bR^2 \l| P_{k,l}(x,y)=0,~~\pa_x  P_{k,l}(x,y) \cdot \pa_y  P_{k,l}(x,y)=0 \r. \r\},$$
where the union is taken over all factors appearing in (\ref{decomposition of qk}) such that both $\pa_x  P_{k,l}$ and $\pa_y  P_{k,l}$ are non-zero polynomials. In other words, $P_{k,l}$ is not a polynomial in only one variable. The equation $\pa_x  P_{k,l}(x,y) \cdot \pa_y  P_{k,l}(x,y)=0$ implies that either $\pa_x  P_{k,l}(x,y)=0$ or $\pa_y  P_{k,l}(x,y)=0$.

We also define $\Gamma_3$ by
$$ \Gamma_3=\bigcup \l\{ (x,y)\in \bR^2 \l| P_{k_1,l_1}(x,y)=P_{k_2,l_2}(x,y)=0 \r. \r\}, $$
where the union is taken over all polynomial factors $P_{k_1,l_1}$ and $P_{k_2,l_2}$, appearing in the factorization (\ref{decomposition of qk}), such that $P_{k_1,l_1}$ and $P_{k_2,l_2}$ are relatively prime.

It should be pointed out that if some $P_{k,l}$ has an isolated zero, say $(x_0,y_0)$, then it follows from the implicit function theorem that $\pa_x P_{k,l} (x_0,y_0) =0 $ and $\pa_y P_{k,l} (x_0,y_0) =0 $.

By B\'{e}zout's theorem, we see that $\Gamma_2$ and $\Gamma_3$ have finitely many points, and the number of points in $\Gamma_2$ and $\Gamma_3$ is bounded in terms of $r, n$ and $\omega$. Similarly, by the Fundamental theorem of algebra, it is clear that $\Gamma_1$ consists of finitely many straight lines parallel in the axes. The number is also bounded in terms of $r, n$, and $\omega$.

Let $Z(P_{k,l})$ be the zeros of $P_{k,l}$, i.e.,
$Z(P_{k,l}) =\l\{ (x,y)\in \bR^2 \l|  P_{k,l}(x,y)=0 \r. \r\}.$
Then $Z(P_{k,l})$ is a set of measure zero. Let $\cL$ be the collection of all $x$-parallel and $y$-parallel straight lines, which intersect the union $\Gamma_1 \bigcup \Gamma_2 \bigcup \Gamma_3$, or which contains the intersection of the boundary of $[0,1]^2$ and $\bigcup ( Z(P_{k,l}) )$. Then the number of straight lines in $\cL$ is bounded in terms of $r, n$, and $\omega$.  Then set $Z$ in Lemma \ref{decomposition of algebraic domains} can be defined by $ Z= \bigcup \l( Z(P_{k,l}) \r) \bigcup \cL.$
Then we can prove that there exists finitely many curved trapezoids $\Omega_1, \Omega_2, \cdots, \Omega_M$ such that
$$ D= \l(\bigcup_{i=1}^M \Omega_i \r) \bigcup Z.$$

Consider the following trilinear sublevel set operator
$$ K_\mu (f,g,h)=\int_D \chi_{\{|H(x,y)|\le \mu \} } f(x) g(y) h(x+y)\dif x \dif y,$$
where $D$ is an algebraic domain in $U$, and $\chi_{\{|H(x,y)|\le \mu \} }$ is the characteristic function of the sublevel set $ \{ (x,y) \in \bR^2 | |H(x,y)| \le \mu \}$. The function $H(x,y)$ is a polynomial in $x$ and $y$.

Our main result for $K_\mu$ is the following theorem.
\begin{theorem}\label{Th-main-2}
Let $H\in \bR[x,y]$ be a polynomial in $x$ and $y$ with real coefficients. Assume $\alpha^{(1)}, \alpha^{(2)}, \cdots, \alpha^{(N)}\in \bN^2\setminus \{(0,0)\}$ and define an algebraic domain by
$$ D=\l\{ (x,y) \in [0,1]^2 : | \pa^{\alpha^{(i)}} H(x,y)| \ge 1, 1\le i \le N \r\}.$$
Then there exists a constant $C>0$, depending only on $\deg H$, such that
\begin{equation}\label{Th-main-2-eq}
|K_\mu(f,g,h)| \le C \mu^{\frac{1}{2d}} \|f\|_2 \|g\|_2 \|h\|_2,
\end{equation}
where $d=\min \{|\alpha^{(i)}|: 1\le i \le N \}$.
\end{theorem}

\begin{remark}
We use the notation
$ \pa^{\alpha^{(i)}} H(x,y)= \pa_x^{\alpha^{(i)}_1} \pa_y^{\alpha^{(i)}_2} H(x,y),$
and $|\alpha^{(i)}|=\alpha^{(i)}_1+\alpha^{(i)}_2$ is the order of $\alpha^{(i)}$.
\end{remark}

\begin{proof}
For each $1\le i \le N$, we shall prove (\ref{Th-main-2-eq}) with $d$ replaced by
$|\alpha^{(i)}|$.\\

\noindent {\bf Case 1. $D= \{ (x,y) \in [0,1]^2: x^{\alpha^{(i)}_1} y^{\alpha^{(i)}_2} \le \mu \}. $}\\

\noindent In this case, $H(x,y)=x^{\alpha^{(i)}_1} y^{\alpha^{(i)}_2}$ and $\pa_x^{\alpha^{(i)}_1} \pa_y^{\alpha^{(i)}_2} H(x,y) = \alpha^{(i)}_1 ! \alpha^{(i)}_2 ! \geq 1$. It is clear that $D=U \bigcap \{ |H(x,y)| \le \mu \}$. Here $\mu$ is a positive number.

Let $j_\mu=\sup\{j\in \bZ: 2^j \le \mu \} $. Then $\mu \in [2^{j_\mu}, 2^{j_\mu+1})$. Hence we have
$$ \int_D |f(x)g(y)h(x+y)| \dif x \dif y \le \sum_{j \le j_\mu} \int_{D_j} |f(x)g(y)h(x+y)| \dif x \dif y,$$
where $D_j=\{(x,y)\in [0,1]^2: 2^j < x^{\alpha^{(i)}_1} y^{\alpha^{(i)}_2} \le 2^{j+1}\}$ for $j \in \bZ$.

If one component of $\alpha^{(i)}$ is equal to zero, say $\alpha^{(i)}_1=0$, then by Cauchy-Schwarz's inequality and Fubini's theorem,
\begin{align*}
\int_{D_j} |f(x)g(y)h(x+y)| \dif x \dif y \le & \l( \int_{D_j} |f(x)|^2 \dif x \dif y \r)^{\frac{1}{2}} \cdot \l( \int_{D_j} |g(y) h(x+y)|^2 \dif x \dif y \r)^{\frac{1}{2}} \\
\le & C 2^{\frac{j}{2\alpha^{(i)}_2}} \|f\|_2 \|g\|_2 \|h\|_2.
\end{align*}
It follows immediately that
$$ \int_D |f(x)g(y)h(x+y)| \dif x \dif y  \le C \mu^{\frac{1}{2\alpha^{(i)}_2}} \|f\|_2 \|g\|_2 \|h\|_2.$$

Now assume that $\alpha^{(i)}_1>0$ and $\alpha^{(i)}_2>0$. By a dyadic decomposition, we can see that
$$ \int_{D_j} |f(x)g(y)h(x+y)| \dif x \dif y \le \sum \int_{\{(x,y)\in [0,1]^2:~2^k \le x^{\alpha^{(i)}_1} \le 2^{k+1}, 2^l \le y^{\alpha^{(i)}_2} \le 2^{l+1}    \}} |f(x)g(y)h(x+y)| \dif x \dif y,$$
where the summation is taken over all $k,l \in \bZ$ satisfying $k,l \le 0$ and $j-1 \le k+l \le j+1$.
By the almost orthogonality principle in Lemma \ref{almost orthogonality}, we have
\begin{align*}
& \sup_{\|f\|_2, \|g\|_2, \|h\|_2 \le 1} \int_{D_j} |f(x)g(y)h(x+y)| \dif x \dif y \\
\le & \sup_{\|f\|_2, \|g\|_2, \|h\|_2 \le 1} \sup_{k,l} \int_{\{(x,y)\in [0,1]^2:~2^k \le x^{\alpha^{(i)}_1} \le 2^{k+1}, 2^l \le y^{\alpha^{(i)}_2} \le 2^{l+1}    \}} |f(x)g(y)h(x+y)| \dif x \dif y \\
\le & C \min\{2^{\frac{k}{2\alpha^{(i)}_1}}, 2^{\frac{l}{2\alpha^{(i)}_2}}\} \\
\le & C 2^{\frac{j}{2|\alpha^{(i)}|}}.
\end{align*}
where the supremum $\sup_{k,l}$ is taken over all negative integers $k,l $ satisfying $j-1 \le k+l \le j+1$.
Taking summation over all $j \le j_\mu$, we obtain the desired $L^2$ estimate. Since the inequality is true with $d$ replaced by $\alpha^{(i)}$ for each $1 \le i \le N$, it also holds for the growth rate $\frac{1}{d}$.\\

\noindent {\bf Case 2. $D$ is a general algebraic domain.}\\

\noindent By Lemma \ref{decomposition of algebraic domains}, we can decompose $D$ into finitely many curved trapezoids, up to a set of measure zero. Assume
$$ D=\l(\bigcup_{i=1}^M \Omega_i \r) \bigcup Z,$$
where each $\Omega_i$ is a curved trapezoid, and $Z$ has measure zero. By a similar argument as Phong, Stein, and Sturm \cite{PSS2001}, one has $a^{\alpha^{(i)}_1} b^{\alpha^{(i)}_2}\le \mu$ for each coordinate rectangle $R\in \Omega_j$ of dimensions $a$ by $b$. Moreover, as in \cite{PSS2001}, we can decompose each $\Omega_j$ into at most countable many curved right triangles, which satisfy the orthogonality assumption in Lemma \ref{almost orthogonality}. This reduces our estimate in Case 2 to Case 1.

Combining all above results, we have completed the proof of Theorem \ref{Th-main-2}.
\end{proof}

Now consider
$$\Lambda_\lambda(f,g,h)=\int \int e^{i\lambda S(x,y)}f(x)g(y)h(x+y) \varphi(x,y) \dif x \dif y ,$$
where $S\in \bR[x,y]$ is a polynomial in $x$ and $y$, and $\varphi(x,y)$ is a smooth cut-off near the origin. The problem then is the uniform estimate for $\Lambda_{\lambda}(f,g,h)$ as $\lambda \to \infty$. For convenience, define a linear oscillatory integral operator by
$$ T_\lambda f(x)= \int_{-\infty}^{\infty} e^{i\lambda P(x,y)} f(y) \varphi(x,y) \dif y,$$
where $P$ is a real-valued polynomial, and $\varphi\in C_0^\infty(\bR^2)$ is supported in a neighborhood of the origin. We need the following operator van der Corput lemma for $T_\lambda$; see \cite{PSS2001} for its proof.
\begin{lem}\label{van der Corput for T}
{\rm (\cite{PSS2001}~)}
Let $P(x,y)\in \bR[x,y]$ be a real polynomial in $x$ and $y$, $\Omega$ be a curved trapezoid in $\bR^2$, and for each $x\in (a,b)$, the cut-off function $\varphi(x,\cdot) \in C_0^\infty \l([g(x), h(x)]\r)$. Here $\Omega$ is given by
$$ \Omega=\l\{ (x,y)\in \bR^2 \l| a<y<b, g(x)<y<h(x) \r. \r\},$$
where $g(x)<h(x)$ for each $x\in (a,b)$. Assume the following conditions are true:\\
{\rm (i)} there exists two positive constants $\mu>0$ and $A\ge 1$ such that
$$ \mu \le S_{xy}^{\prime \prime} (x,y) \le A \mu, \quad (x,y)\in \Omega,$$
{\rm (ii)} for some $B>0$, it is true that
$$ \sup_\Omega \sum_{k=0}^{2} (\tau(x))^k | \pa_y^k \varphi(x,y)| \le B, $$
where $\tau(x)$ is the length of the cross section $\{y\in \bR |(x,y) \in \Omega\}. $
Define
$$ T_{\lambda,\Omega} f(x)= \int_{-\infty}^{\infty} e^{i\lambda S(x,y)} \chi_{\Omega}(x,y) \varphi(x,y) f(y) \dif y.$$
Then we have
$$ \|T_{\lambda,\Omega} f \|_2 \le C |\lambda \mu|^{-\frac{1}{2}} \|f\|_2,$$
where the constant depends only on $\deg S, A$ and $B$.
\end{lem}

\noindent As a consequence of Lemma \ref{van der Corput for T}, we have the following operator van der Corput lemma for trilinear oscillatory integrals.

\begin{lem}\label{van der Corput for trilinear}
Assume $S\in \bR[x,y]$ is a real polynomial in $x$ and $y$, and $\Omega$ is a curved trapezoid given by
$$ \Omega=\l\{ (x,y)\in \bR^2 \l| a<x<b, \alpha(y)<x<\beta(y) \r. \r\},$$
where $\alpha$ and $\beta$ are continuous monotone functions on $[a,b]$ such that $\alpha(y)<\beta(y)$ for all $y\in (a,b)$. Let $H(x,y)=\pa_x \pa_y (\pa_x-\pa_y)S(x,y).$ If the following conditions are true:\\

\noindent {\rm (i)} for each $y \in (a,b)$, $\varphi(\cdot,y)\in C_0^\infty\l([\alpha(y), \beta(y)]\r)$ and there exists a constant $B>0$ such that
$$ \sup_\Omega \sum_{k=0}^{2} (\tau(y))^k |\pa_x^k \varphi(x,y)| \le B,$$
where $\tau(y)$ is the length of the cross section $\{x\in\bR| (x,y)\in \Omega\}$;\\
{\rm(ii)} there exist two positive constants $\mu>0$ and $A\le 1$ such that
$$ \mu \le |H(x,y)| \le A\mu, \quad (x,y)\in \Omega,$$
then we have
$$ |\Lambda_{\lambda,\Omega} (f,g,h)| \le C|\lambda \mu|^{-\frac{1}{6}} \|f\|_2 \|g\|_2\|h\|_2,$$
where the constant $C>0$ depends only on $\deg S, A$ and B, and $\Lambda_{\lambda,\Omega} (f,g,h)$ is defined by
$$\Lambda_{\lambda,\Omega}(f,g,h)=\int \int e^{i\lambda S(x,y)}f(x)g(y)h(x+y) \chi_{\Omega}(x,y) \varphi(x,y) \dif x \dif y .$$
\end{lem}

\begin{proof}
Let $T_\lambda (g,h)$ be defined by
$$ T_\lambda (g,h)(x)= \int_{-\infty}^{\infty} e^{i\lambda S(x,y)}g(y) h(x+y) \chi_{\Omega}(x,y) \varphi(x,y) \dif y.$$
Then it is clear that $\Lambda_{\lambda}(f,g,h)=\int_{\bR} T_\lambda(g,h)(x)f(x)\dif x$.

Now we compute the $L^2$ norm of $T_\lambda(g,h)$. Set $\varphi_{\Omega}(x,y)=\varphi(x,y)\chi_{\Omega}(x,y)$. Then we have
\begin{align}
&\int_{\bR} |T_\lambda (g,h)(x)|^2 \dif x \notag\\
=&\int_{\bR^3} e^{i\lambda[S(x,y)-S(x,z)]} g(y) h(x+y) \overline{g(z) h(x+z)} \varphi_{\Omega}(x,y) \overline{\varphi_{\Omega}(x,z)} \dif y \dif z \dif x \notag \\
=& \int_{\bR^3} e^{i\lambda[S(x,y)-S(x,y+u)]} g(y) \overline{g(y+u)} h(x+y) \overline{ h(x+y+u)} \varphi_{\Omega}(x,y) \overline{\varphi_{\Omega}(x,y+u)}\dif x \dif y\dif u \notag\\
=& \int_{\bR} \l[ \int_{\bR^2}     e^{i\lambda[S(v-y,y)-S(v-y,y+u)]} \w{g}_u(y) \w{h}_u(v)  \varphi_{\Omega}(v-y,y) \overline{\varphi_{\Omega}(v-y,y+u)}\dif v \dif y \r]  \dif u \label{eq3-3}
\end{align}
where we have made changes of variables $z=y+u, v=x+y$, and
\begin{eqnarray*}
\w{S}_u(v,y)&= &S(v-y,y)-S(v-y,y+u),\\
\w{g}_u(y)&= &g(y) \overline{g(y+u)}, \\
\w{h}_u(v)&= & h(v) \overline{ h(v+u)},\\
\w{\varphi}_u(v,y)&= &\varphi(v-y,y) \overline{\varphi(v-y,y+u)},
\end{eqnarray*}
for each $u\in \bR$. Then the inner integral with respect to $v$ and $y$ in (\ref{eq3-3}) can be written as
\begin{equation}\label{eq3-4}
\int_{\bR^2} e^{i\lambda \w{S}_u(v,y) }  \w{g}_u(y) \w{h}_u(v) \w{\varphi}_u(v,y) \chi_{\w{\Omega}_u}(v,y) \dif v \dif y.
\end{equation}
Let $\w{\Omega}_u$ be
\begin{align*}
\w{\Omega}_u =& \l\{(v,y)\in \bR^2 \l| a<y<b, a<y+u<b, \alpha(y)< v-y <\beta(y), \alpha(y+u)< v-y <\beta(y+u) \r. \r\} \\
=& \l\{(v,y)\in \bR^2 \l| \w{a}_u<y<\w{b}_u, \w{\alpha}_u(y)< v <\w{\beta}_u(y) \r. \r\},
\end{align*}
where
\begin{eqnarray*}
\w{a}_u&=&\max\{a,a-u\}, \\
\w{b}_u&=&\min\{b,b-u\}, \\
\w{\alpha}_u(y)&=&y+\max\{\alpha(y), \alpha(y+u)\}, \\
\w{\beta}_u(y)&=&y+\min\{\beta(y), \beta(y+u)\}.
\end{eqnarray*}
Assume $\w{a}_u<\w{b}_u$. Then both $\w{\alpha}_u$ and $\w{\beta}_u$ are monotone continuous functions on $[\w{a}_u, \w{b}_u]$. And for each $y\in (\w{a}_u, \w{b}_u)$, the cut-off function $\w{\varphi}(\cdot, y)$ is a smooth function supported in the horizontal cross section $\{v \in \bR | (v,y)\in \w{\Omega}_u \}$. Moreover,
\begin{align*}
 \sup_{\w{\Omega}_u} \sum_{k=0}^{2} (\w{\tau}_u(y))^k |\pa_v^k \w{\varphi}_u (v,y)|
\le  \l(\sup_\Omega \sum_{k=0}^{2} (\tau(y))^k |\pa_x^k \varphi (x,y)|  \r)^2
\le  B^2.
\end{align*}
Here $\w{\tau}_u(y)$ is the length of the interval $\{v\in \bR | (v, y)\in \w{\Omega}_u \}$.

On the other hand, it is clear that
$$ \pa_v \pa_y \l[S(v-y,y)-S(v-y,y+u)\r]= u \int_{0}^{1} H(v-y, y+\theta u) \dif \theta.$$
Hence, by our assumption (ii),
$$ \mu |u| \le \l|\pa_v \pa_y \l[S(v-y,y)-S(v-y,y+u)\r]\r|\le A\mu |u|.$$
By Lemma \ref{van der Corput for T}, the integral in (\ref{eq3-4}) is bounded by
\begin{align*}
& \l|\int_{\bR^2} e^{i\lambda \w{S}_u(v,y) }  \w{g}_u(y) \w{h}_u(v) \w{\varphi}_u(v,y) \chi_{\w{\Omega}_u}(v,y) \dif v \dif y \r| \\
\le & C\l(|\lambda|\mu |u|\r)^{-\frac{1}{2}} \|\w{g}_u(\cdot)\|_2 \|\w{h}_u(\cdot)\|_2,
\end{align*}
where $C>0$ depends only on $\deg S, A$, and $B$. Taking the absolute value into the integral, we also have
\begin{align*}
& \l|\int_{\bR^2} e^{i\lambda \w{S}_u(v,y) }  \w{g}_u(y) \w{h}_u(v) \w{\varphi}_u(v,y) \chi_{\w{\Omega}_u}(v,y) \dif v \dif y \r| \\
\le & C \|\w{g}_u(\cdot)\|_1 \|\w{h}_u(\cdot)\|_1,
\end{align*}
where $C>0$ depends only on $B$. By the Cauchy-Schwarz inequality, $\|\w{g}_u(\cdot)\|_1\le \|g\|_2^2$ and also $\|\w{g}_u(\cdot)\|_1 \le \|h\|_2^2$.

Now we are able to estimate $\|T_\lambda (g,h) \|_2^2$. In fact, we make use of the above $L^2$ and $L^1$ estimates together with Cauchy-Schwarz's inequality to obtain
\begin{align*}
\|T_\lambda (g, h) \|_2^2 \le & \int_{\bR} \l| \int_{\bR^2} e^{i\lambda \w{S}_u(v,y) }  \w{g}_u(y) \w{h}_u(v) \w{\varphi}_u(v,y) \chi_{\w{\Omega}_u}(v,y) \dif v \dif y \r| \dif u \\
= & \sum_{j\in \bZ} \int_{2^j}^{2^{j+1}} \l| \int_{\bR^2} e^{i\lambda \w{S}_u(v,y) }  \w{g}_u(y) \w{h}_u(v) \w{\varphi}_u(v,y) \chi_{\w{\Omega}_u}(v,y) \dif v \dif y \r| \dif u \\
\le & C \sum_{j\in \bZ} \min\{\l(|\lambda \mu| 2^j \r)^{-\frac{1}{2}}, 2^j  \} \|g\|_2^2 \|h\|_2^2 \\
\le & C |\lambda \mu|^{-\frac{1}{3}} \|g\|_2^2 \|h\|_2^2,
\end{align*}
where $C$ depends only on $\deg S, A$ and $B$. This implies $\|T_\lambda (g, h) \|_2 \le C |\lambda \mu|^{-\frac{1}{6}} \|g\|_2 \|h\|_2$. By duality, the estimate for $\Lambda_{\lambda, \Omega}(f,g,h)$ follows immediately. This completes the proof of Lemma \ref{van der Corput for trilinear}.
\end{proof}

Now we can state our main result in this section.
\begin{theorem}\label{Th-main-3}
Assume $S(x,y)\in \bR[x,y]$ is a real polynomial. Let $H(x,y)=\pa_x \pa_y (\pa_x -\pa_y) S(x,y)$, and $\alpha^{(1)}, \alpha^{(2)}, \cdots, \alpha^{(N)}\in \bN^2$. Suppose that $\varphi \in C_0^\infty(U)$ and
$$| \pa^{\alpha^{(i)}} H(x,y)| \ge 1, \quad 1\le i \le N,$$
on the support of $\varphi$. Then $\Lambda_\lambda(f,g,h)$ satisfies the inequality
$$  |\Lambda_\lambda(f,g,h)| \le C |\lambda|^{-\frac{1}{2(3+d)}} \|f\|_2 \|g\|_2 \|h\|_2$$
with $d=\min\{|\alpha^{(i)}|: 1\le i\le N\} $. Here the constant $C$ depends only on $\deg S$ and the cut-off function $\varphi$.
\end{theorem}

\begin{proof}
Choose a smooth function $\phi \in C_0^\infty \l([\frac12, 2]\r)$ such that $\sum_{j \in \bZ} \phi(\frac{x}{2^j})=1$ for all $x>0$. Let $\Lambda_\lambda^{(j)}(f,g,h)$ be defined as $\Lambda_\lambda(f,g,h)$ with insertion of $\phi\l(\frac{|H(x,y)|}{2^j}\r)$ in the cut-off. In other words,
$$ \Lambda_\lambda^{(j)}(f,g,h)= \int \int e^{i\lambda S(x,y)}f(x)g(y)h(x+y) \phi\l(\frac{|H(x,y)|}{2^j}\r) \varphi(x,y) \dif x \dif y,~~~j\in\bZ. $$

Now consider the simplest case $d=0$. By removing only the horizontal lines appearing in the proof of Lemma \ref{decomposition of algebraic domains}, we can decompose the following algebraic domain
$$ D_j= \l\{ (x,y)\in U \l| 2^{j-1}\leq |\pa^{\alpha^{(i)}} H(x,y)| \leq 2^{j+1} \r. \r\}$$
into finitely many curved trapezoids described in Lemma \ref{van der Corput for trilinear}. Moreover, the cut-off function of $\Lambda_\lambda^{(j)}(f,g,h)$, with the vertical variable fixed, is compactly supported in the corresponding horizontal cross section of each curved trapezoid mentioned above. Hence we can apply Lemma \ref{van der Corput for trilinear} to obtain $ \l|\Lambda_\lambda^{(j)}(f,g,h)\r| \le C \l(|\lambda|2^j \r)^{-\frac{1}{6}} \|f\|_2 \|g\|_2 \|h\|_2$ for $j\geq 0$. Taking summation over $j$, we obtain the desired estimate.

Now we focus on the general case $d>0$. By Theorem \ref{Th-main-2}, we obtain
$$ \l|\Lambda_\lambda^{(j)}(f,g,h)\r| \le C 2^{\frac{j}{2d}} \|f\|_2 \|g\|_2 \|h\|_2.$$
On the other hand, we can apply Lemma \ref{van der Corput for trilinear} to get
$$ \l|\Lambda_\lambda^{(j)}(f,g,h)\r| \le C \l(|\lambda|2^j \r)^{-\frac{1}{6}} \|f\|_2 \|g\|_2 \|h\|_2.$$
Hence it follows that
\begin{align*}
\l|\Lambda_\lambda(f,g,h)\r| \le & \sum_{j\in \bZ} \l|\Lambda_\lambda^{(j)}(f,g,h)\r| \\
\le & C \sum_{j\in \bZ} \min\{2^{\frac{j}{d}}, \l(|\lambda|2^j \r)^{-\frac{1}{6}} \} \|f\|_2 \|g\|_2 \|h\|_2 \\
\le & |\lambda|^{-\frac{1}{2(3+d)}} \|f\|_2 \|g\|_2 \|h\|_2.
\end{align*}
This completes the proof of the theorem.
\end{proof}

\noindent Inspired by an observation in Gressman \cite{gressman2} for bilinear oscillatory integral operators, we can prove the following theorem by a similar argument in this section.
\begin{theorem}
Assume $S\in \bR[x,y]$ is a real polynomial, and $H(x,y)=\pa_x \pa_y (\pa_x -\pa_y) S(x,y)$. If $\phi \in C_0^\infty \l([\frac12, 2]\r)$, then
\begin{align*}
& \l|\int_{[0,1]^2} e^{i\lambda S(x,y)} f(x)g(y)h(x+y) \phi\l(\frac{H(x,y)}{\mu} \r) \dif x \dif y \r| \\
\le & C |\lambda \mu|^{-\frac{1}{6}} \|f\|_2 \|g\|_2 \|h\|_2
\end{align*}
for all $\mu \in \bR \setminus \{0\}$, where the constant $C$ depends only on $\deg S$ and the cut-off function $\phi$.
\end{theorem}

\end{document}